\documentclass[11pt,preprint]{amsart}

\usepackage{amsfonts}
\RequirePackage{natbib}
\usepackage{graphicx}
\usepackage{color}
\usepackage{amssymb}
\usepackage{amsfonts}
\usepackage{amsthm}
\usepackage[mathcal]{eucal}
\usepackage{ipa}
\usepackage{dsfont}

\usepackage[justification=centering,font=footnotesize]{caption}
\usepackage[labelformat=simple]{subfig}

\newtheorem{theorem}{Theorem}[section]
\newtheorem{lemma}[theorem]{Lemma}
\newtheorem{proposition}[theorem]{Proposition}

\newtheorem{remark}[theorem]{Remark}
\newtheorem{corollary}[theorem]{Corollary}

\newtheorem{example}[theorem]{Example}

\def\R{\mathbb R}
\def\N{\mathbb N}

\newcommand{\begitem}{\begin{itemize}}
\newcommand{\finit}{\end{itemize}}

\newcommand\restr[2]{{
  \left.\kern-\nulldelimiterspace 
  #1 
  \vphantom{\big|} 
  \right|_{#2} 
  }}

\newcommand{\bbe}{\mathbf{e}}

\newcommand{\vphi}{\varphi}

\newcommand{\E}{\mathbb{E}}
\newcommand{\Prob}{\mathbb{P}}
\newcommand{\ProbQ}{\mathbb{Q}}

\newcommand{\eps}{\varepsilon}

\begin{document}

\title[Allee effects and environmental stochasticity]{Pushed beyond the brink: Allee effects, environmental stochasticity, and extinction}
\author[G. Roth]{Gregory Roth}
\email{greg.roth51283@gmail.com}
\author[S.J. Schreiber]{Sebastian J. Schreiber}
\email{sschreiber@ucdavis.edu}
\address{Department of Evolution and Ecology, One Shields Avenue, University of California, Davis, CA 95616 USA}

\maketitle

\begin{abstract}
A demographic Allee effect occurs when individual fitness, at low densities, increases with population density. Coupled with environmental fluctuations in demographic rates, Allee effects can have subtle effects on population persistence and extinction. To understand the interplay between these deterministic and stochastic forces, we analyze discrete-time single species  models allowing for general forms of density-dependent feedbacks and stochastic fluctuations in demographic rates. Our analysis provide criteria for stochastic persistence (the population tends to remain bounded away from extinction for all positive initial conditions), asymptotic extinction (the population tends to extinction for all initial conditions), and conditional persistence (the population persists or goes extinct with positive probability for some initial conditions). Stochastic persistence requires that the geometric mean of fitness at low densities is greater than one. When this geometric mean is less than one, asymptotic extinction occurs with a high probability whenever the initial population density is low.  If in addition the population only experiences positive density-dependent feedbacks,  conditional persistence occurs provided the geometric mean of fitness at high population densities is greater than one. However, if the population experiences both positive and negative density-dependent feedbacks, conditional persistence is only possible if fluctuations in demographic rates are sufficiently small. Applying our results to stochastic models of mate-limitation, we illustrate counter-intuitively that the environmental fluctuations can increase the probability of persistence when populations are initially at low densities, and decrease the likelihood of persistence when populations are initially at high densities. Alternatively, for stochastic models accounting for predator saturation and negative density-dependence, environmental stochasticity can result in asymptotic extinction at intermediate predation rates despite conditional persistence occurring at higher predation rates. 
\end{abstract}

\section{Introduction}

Populations exhibit an Allee effect when at low densities individual fitness increases with density~\citep{allee-31,stephens-etal-99}. Common causes of this positive density-dependent feedback include predator saturation, cooperative predation, increased availability of mates, and conspecific enhancement of reproduction~\citep{courchamp-etal-99,stephens-etal-99,gascoigne-lipcius-04,courchamp-etal-08,gascoigne-etal-09,kramer-etal-09}. When an Allee effect is sufficiently strong, it can result in a critical density below which a population is driven rapidly to extinction through this positive feedback. Consequently, the importance of  the Allee effect has been widely recognized for conservation of at risk populations~\citep{dennis-89,stephens-sutherland-99,berec-etal-07,courchamp-etal-08} and management of invasive species~\citep{keitt-etal-01,leung-etal-04,tobin-etal-11}. Population experiencing environmental stochasticity and a strong Allee effect are widely believed to be especially vulnerable to extinction as the fluctuations may drive their densities below the critical threshold~\citep{courchamp-etal-99,dennis-02,berec-etal-07,courchamp-etal-08}. However, unlike the deterministic case~\citep{cushing-88,gyllenberg-etal-96,tpb-03,yakubu-03,luis-etal-10,kang-lanchier-11,duarte-etal-12}, the mathematical theory for populations simultaneously experiencing an Allee effect and environmental stochasticity is woefully underdeveloped (see, however, \citet{dennis-02}). 

To better understand the interplay between Allee effects and environmental stochasticity, we examine stochastic, single species models of the form 
\begin{equation}\label{eq:model}
X_{t+1}= f(X_t, \xi_{t+1}) X_t
\end{equation}
where $X_t\in [0,\infty)$ is the density of the population at time $t$, $f(x,\xi)$ is the fitness of the population as a function of its density and the environmental state $\xi$, and the environmental fluctuations $\xi_t$ are given by a sequence of independent and identically distributed (i.i.d.) random variables. Here we determine when these deterministic and stochastic forces result in unconditional stochastic persistence (i.e. the population tends to stay away from extinction for all positive initial conditions with probability one), unconditional extinction (i.e. the population tends asymptotically to extinction with probability one for all initial conditions), and conditional stochastic persistence (i.e. the population persists with positive probability for some initial conditions and goes extinct with positive probability for initial conditions). Section 2 describes our standing assumptions. Section 3  examines separately how negative-density dependence and positive-density dependence interact with environmental stochasticity to determine these different outcomes. For models with negative density-dependence (i.e. $f(x,\xi)$ is a decreasing function of density $x$),  \citep{jdea-11} proved that generically, these models only can exhibit unconditional persistence or unconditional extinction. For models with only positive density-dependence (i.e. $f(x,\xi)$ is an increasing function of density $x$), we prove that all three dynamics (unconditional persistence, unconditional extinction, and conditional persistence) are possible and provides a characterization for these outcomes.  Section 4 examines the combined effects of negative and positive density dependence on these stochastic models. We prove that conditional persistence only occurs when the environmental noise is ``sufficiently'' small. We illustrate the main results using models for mate limitation and predator saturation. Section 5 concludes with a discussion of the implications of our results, how these results relate to prior results, and identifying future challenges. 

\section{Models, assumptions, and definitions}
Throughout this paper, we study stochastic difference equations of the form given by equation~\eqref{eq:model}. For these equations, we make \emph{two standing assumptions} 
\begin{description}
\item [Uncorrelated environmental fluctuations]  $\{\xi_t\}_{t=0}^\infty$ is a sequence of independent and identically distributed (i.i.d) random variables taking values in a separable metric space $E$ (such as $\R^n$).
\item [Fitness depends continuously on population and environmental state] the fitness function $f:\R_+\times E \to \R_+$ is  continuous on the product of the non-negative half line $\R_+=[0,\infty)$ and the environmental state space $E$.
\end{description}
The first assumption implies that $(X_t)_{t\ge 0}$ is a Markov chain on the population state space $\R_+$. While we suspect our results hold true without this assumption, the method of proof becomes more difficult and will be considered elsewhere. The second assumption holds for most population models.  

Our analysis examines conditions for asymptotic extinction (i.e. $\lim_{t\to\infty}X_t=0$) occurring with positive probability and persistence (a tendency for populations to stay away from extinction) with positive probability. Several of our results make use of the \emph{empirical measures} for the Markov chain $(X_t)_{t\ge 0}$ given by
\[
\Pi_t = \frac{1}{t}\sum_{s=0}^{t-1} \delta_{X_s}
\]
where $\delta_x$ denotes a Dirac measure at the point $x$ i.e. $\delta_x(A)=1$ if $x\in A$ and $0$ otherwise.  For any interval $[a,b]$ of population densities, $\Pi_t([a,b])$ is the fraction of time that the population spends in this interval until time $t$. The long-term frequency  that $(X_t)_{t\ge 0}$ enters the interval $[a,b]$ is given by $\lim_{t\to\infty} \Pi_t([a,b])$, provided the limit exists. As these empirical measures depend on the stochastic trajectory, they are random probability measures. 

\section{Negative- versus positive-density dependence}

\subsection{Results for negative-density dependence}
For models with only the negative density dependence (i.e. fitness $f$ is a decreasing function of density), the dynamics of the model \eqref{eq:model} exhibit one of three possible behaviors: asymptotic extinction with probability one, unbounded population growth with probability one, or stochastic persistence and boundedness with probability one. To state this result, recall that $\log^+x=\max\{\log x, 0\}$.

\begin{theorem}[Schreiber 2012]\label{thm:scalar} Assume $f(x,\xi)$ is a positive decreasing function in $x$ for all $\xi \in E$ and  $\E[\log^+ f(0,\xi_t)]<\infty$. Then
\begin{description}
\item[Extinction] if $\E[\log f(0,\xi_t)]<0$,  then $\lim_{t\to\infty} X_t=0$ with probability whenever  $X_0=x\ge 0$,
\item[Unbounded growth] if $\lim_{x\to\infty} \E[\log f(x,\xi_t)]>0$,  then $\lim_{t\to\infty} X_t=\infty$ with probability whenever  $X_0=x>0$, and
\item[Stochastic persistence]  if $\E[\log f(0,\xi)]>0$ and $\lim_{x\to\infty} \E[\log f(x,\xi_t)]<0$, then for all $\epsilon>0$ there exists $M>0$ such that 
\[
\limsup_{t\to\infty} \Pi_t( [1/M,M] ) \ge 1-\epsilon \mbox{ almost surely}
\]
whenever $X_0=x>0$.
\end{description}
\end{theorem} 
In the case of stochastic persistence, the typical trajectory spends most of its time in a sufficiently large compact interval excluding the extinction state $0$.

To illustrate Theorem~\ref{thm:scalar}, we apply it to  stochastic versions of the Ricker and Beverton-Holt models. For the stochastic Ricker model, the fitness function is  $f(x,\xi)= \exp(r -a x)$ where $\xi=(r,a)$. Stochasticity in $r_t$ and $a_t$ may be achieved by allowing $r_t$ to be a sequence of i.i.d. normal random variables or $a_t$ to be a sequence of i.i.d. log-normal random variables. These choices always satisfy the assumption $\lim_{x\to\infty} \E[\log f(x,\xi_t)]=-\infty$.  This stochastic Ricker model is almost surely persistent if  $\E[\log f(0,\xi_t)]=\E[r_t]>0$. If $\E[r_t]<0$, then asymptotic extinction occurs with probability one.  

For a stochastic version of the Beverton-Holt model, we have $f(x,\xi)=a/(1+bx)$ with $\xi=(a,b)$. Stochasticity in $a_t$ and $b_t$ may be achieved by allowing them to be sequences of i.i.d. log-normal random variables. These choices always satisfy the assumption $\lim_{x\to\infty} \E[\log f(x,\xi_t)]=-\infty$.  This stochastic Beverton-Holt model is almost surely persistent if  $\E[\log f(0,\xi_t)]=\E[\log a_t]>0$.  If $\E[\log a_t]<0$, then asymptotic extinction occurs with probability one.

\subsection{Results for positive density-dependence}
In contrast to models with only negative density-dependence, models with only positive density-dependence exhibit a different trichotomy of dynamical behaviors: asymptotic extinction for all initial conditions, unbounded population growth for all positive initial conditions, or conditional persistence in which there is a positive probability of the population going asymptotically extinct and a positive probability of unbounded population growth. To characterize this trichotomy, we say $\{0,\infty\}$ is \emph{accessible from the set $B\subset (0,\infty)$}  if for any $M>0$, there exists $\gamma>0$ such that
\[
\Prob\left[\left\{\exists t \ge 0 \ : \ X_t \in [0,1/M]\cup[M,\infty)\right\} \ | \  X_0=x\right]>\gamma
\] 
for all $x\in B$. 

\begin{theorem} \label{thm:PDD}Assume $f(x,\xi)$ is an increasing function of $x$ for all $\xi \in E$. Define $f_\infty (\xi)=\lim_{x\to\infty} f(x,\xi)$. Then 
\begin{description}
\item [Extinction] if $ \E[\log f_\infty (\xi_t)]<0$, then $\lim_{t\to\infty}X_t = 0$ with probability one whenever $X_0=x\ge 0$. 
\item [Unbounded growth] if $\E[\log f(0,\xi_t)]>0$, then $\lim_{t\to\infty}X_t= \infty$ with probability one whenever  $X_0=x> 0$.
\item [Conditional persistence] if $\E[\log f(0,\xi_t)]<0$ and  $ \E[\log f_\infty(\xi_t)]>0$, then for any $0<\delta<1$, there exist $m,M>0$ such that 
\[
\Prob\left[  \lim_{t\to\infty} X_t =\infty\Big | X_0= x\right] \ge 1- \delta \mbox{ and }\Prob\left[\lim_{t\to\infty} X_t =0\Big| X_0=y\right]\ge1-\delta,
\]
for all $x\in[M,\infty)$ and all $y \in (0,m]$.

Moreover, if $\{0,\infty\}$ is accessible, then 
\[
\Prob \left[\left\{ \lim_{t\to\infty} X_t =0 \mbox{ or }\infty\right\} \Big| X_0=x\right]=1
\]
for all $x\in (0,\infty)$.
\end{description}

\end{theorem}

\begin{figure}
\includegraphics[width=0.6\textwidth]{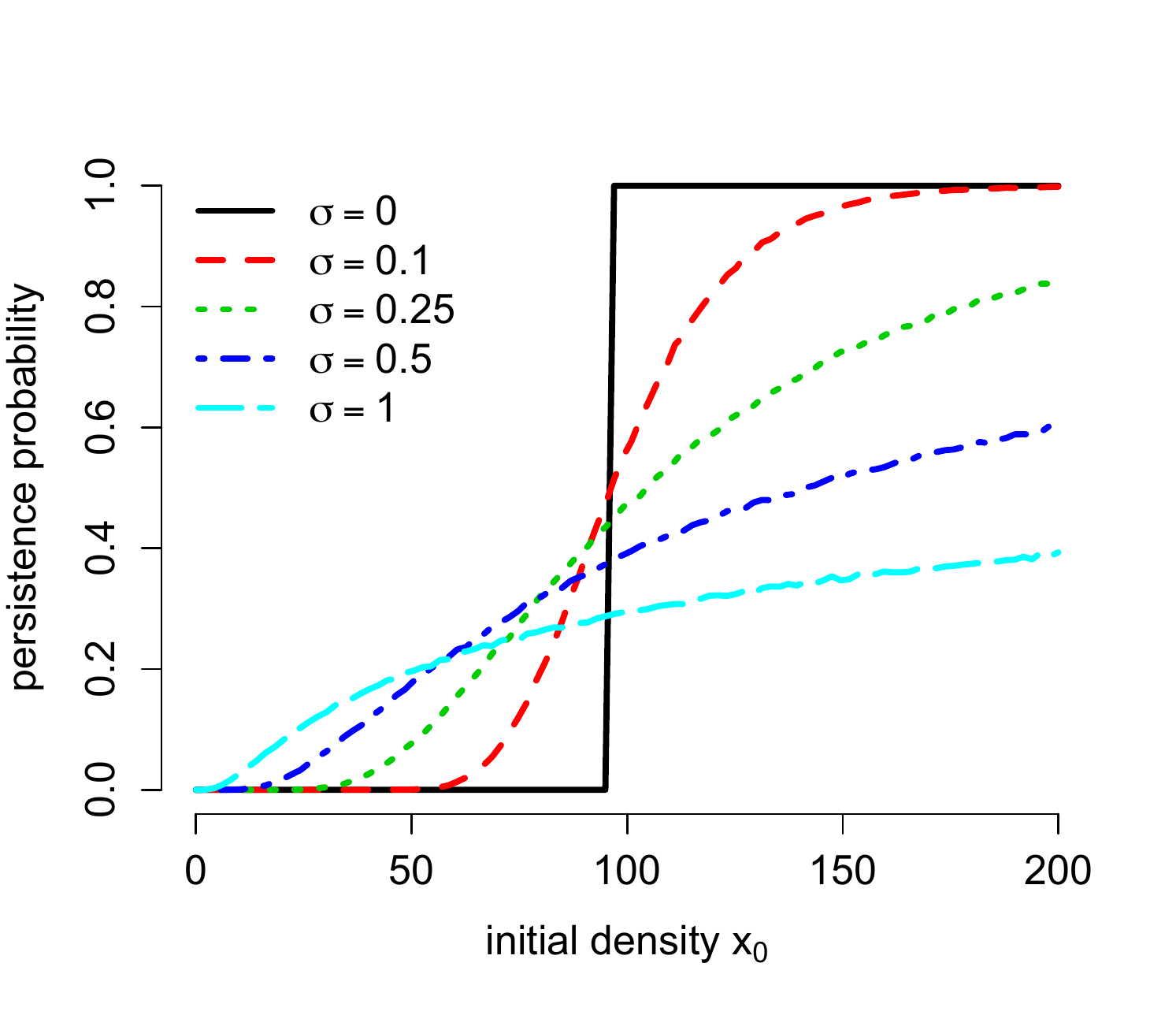}
\caption{Effect of initial population density on persistence for the stochastic mate limitation model. The stochastic mate limitation model with $f(x,\xi)= \frac{\lambda x}{h+x}$ where $\xi =(\lambda, h)$ was simulated $10,000$ times for each initial density. The fraction of runs where the final density was greater than $100$ are plotted as a function of initial density $x_0$. Parameters: $h=10$ and $\lambda$ log-normally distributed with log-normal mean $0.1$ and log-normal standard deviations $\sigma$ as shown. }\label{fig:PDD}
\end{figure}

To illustrate Theorem~\ref{thm:PDD}, we apply it to stochastic versions of models accounting for mate-limitation and a predator-saturation.  For many sexually reproducing organisms, finding mates becomes more difficult at low densities. For instance, pollination of plants by animal vectors becomes less effective when patches become too small because lower densities result is reduced visitation rates by pollinators~\citep{groom-98}. Alternatively,  fertilization by free spawning gametes of benthic invertebrates can become insufficient at low densities~\citep{knowlton-92,levitan-etal-92}.  To model mate limitation, let $x$ be the density of females in the population. Assuming a 50-50 sex ratio (i.e. $x$ also equals the density of males in the population), \citet{dennis-89,mccarthy-97,scheuring-99} modeled  the probability of a female finding a mate by  the function
\[
\frac{x}{h+x}
\]
where $h$ is a half-saturation constant i.e. the male density at which 50\% of the females find a mate. If $\lambda$ is the number of daughters produced per mated female, then the fitness function is
\[
f(x,\xi)= \frac{\lambda x}{h+x} \mbox{ where }\xi =(\lambda, h)
\]
Stochasticity in $\xi_t$ may be achieved by allowing $\lambda_t,h_t$ to be sequences of i.i.d. log-normally distributed random variables. Since $\E[\log f(0,\xi_t)]=-\infty$, this stochastic model always exhibits asymptotic extinction for some initial conditions with positive probability. Theorem~\ref{thm:PDD} implies that asymptotic extinction occurs for all initial conditions with probability one if $\E[\lim_{x\to\infty}f(x,\xi_t)]=\E[\log \lambda_t]<0$. On the other hand, conditional persistence occurs if $\E[\log \lambda_t]>0$. 

Figure~\ref{fig:PDD} illustrate how the probability of  persistence for the mate limitation model depends on initial condition and the level of environmental stochasticity. Interestingly, higher levels of environmental stochasticity promote higher probabilities of persistence when initial population densities are low. Intuitively, when the population is below the ``Allee threshold'', environmental stochasticity provides opportunities of escaping the extinction vortex.

Another common Allee effect occurs in species subject to predation by a generalist predator with a saturating functional response. Within such populations, an individual's risk of predation decreases as the population's density increases.  For example, in field studies, \citet{crawley-long-95} found that per-capita rates of acorn loss of {\em Quercus robur L.} to invertebrate seed predators were greatest (as high as 90\%) amongst low acorn crops  and lower (as low as 30\%) on large acorn crops. To model Allee effects due to predator saturation, \citet{tpb-03} used the following fitness function
\[
f(x,\xi)=\exp\left(r - \frac{P}{h+  x}\right) \mbox{ where } \xi=(r,P,h)
\]
where $r$ is the intrinsic rate of growth of the focal population, $P$ is the predation intensity, and $h$ is a half-saturation constant. Stochasticity may be achieved by allowing $r_t$ to be normally distributed and $h_t,P_t$ be log-normally distributed. Theorem~\ref{thm:PDD}  implies that unbounded growth occurs for all initial conditions  whenever $\E[\log f(0,\xi_t)]=\E[r_t -  P_t]>0$.  Alternatively, $\E[\log f_\infty(\xi_t)]=\E[r_t]<0$ implies asymptotic extinction with probability one for all initial conditions. Conditional persistence occurs when both of these inequalities are reversed. 

\section{Positive and negative density dependence}

For populations exhibiting positive and negative density dependence, the fitness function $f(x,\xi)$ can increase or decrease with density. For these general fitness functions, we prove several results about asymptotic extinction and persistence in the next two subsections. 

\subsection{Extinction}
We begin by showing that assumptions
\begin{description}
\item[ \textbf{A1}] $\E[\log f(0,\xi_t)]<0$, and 
\item[ \textbf{A2}]there exists $\gamma>0$ such that $x\mapsto f(x,\xi)$ is increasing on $[0,\gamma)$ for all $\xi \in E$,
\end{description}
implies asymptotic extinction occurs with positive probability for populations at low densities. Furthermore, we show this asymptotic extinction occurs with probability one for all positive initial conditions whenever the extinction set $\{0\}$ is ``accessible'' i.e. there is always a positive probability of the population density getting arbitrarily small. More specifically, we say  $\{0\}$ is \emph{accessible} from $A\subset [0,\infty)$ if for any $\eps>0$, there exists $\gamma>0$ such that
\[
\Prob[\exists t \ge 0 \ : \ X_t < \eps \ | \  X_0=x]>\gamma
\] 
for all $x\in A$. We call a set $B \subset \R_+$ \emph{invariant} if 
$\Prob\left[X_{1}\in B \big| X_0\in B\right]=1$.  
\begin{theorem} \label{thm:extinction} 
Assume \textbf{A1} and \textbf{A2}. Then for any $\delta>0$, there exists $\epsilon>0$ such that 
\[
\Prob\left[ \lim_{t\to\infty} X_t =0 | X_0 =x\right]\ge 1-\delta
\]
for all $x \in [0,\epsilon]$. Furthermore, if $\{0\}$ is accessible from $[0,M)$ for some $M>0$ (possibly $+\infty$), and $[0,M)$ is invariant, then 
\[
\Prob\left[\lim_{t\to\infty} X_t =0 \vert X_0=x\right]=1
\]
for all $x \in [0,M)$.
\end{theorem}

There are two cases for which one can easily verify accessibility of $\{0\}$. First, suppose that $f(x,\xi)=g(x)\xi$. If $(\xi_t)_{t\le 0}$ is a sequence of log-normal or gamma-distributed i.i.d. random variables and $x\mapsto xg(x)$ is bounded (i.e. there exists $M>0$ such that $xg(x)\le M$ for all $x$). Then, it follows immediately from the definition of accessibility that  $\{0\}$ is accessible from $[0,\infty)$. Hence, in this case $\E[\log f (0,\xi_t)]<0$ implies unconditional extinction. Since log-normal random variables and gamma random variables can take on any positive value, we view this case as the ``large noise'' scenario i.e. there is a positive probability of the log-population size changing by any amount. 

Alternatively, for sufficiently, small noise, there is simple conditions for accessibility of $\{0\}$. Define $F:\R_+ \times E \rightarrow \R_+$ by $F(x,\xi)=f(x,\xi)x$ and the ``unperturbed model'' $F_0:\R_+ \rightarrow \R_+$ by $F_0(x)=F(x,\E[\xi_1])$. For any $x\in \R$, define $x^+=\max\{0,x\}$.  A system \eqref{eq:model} satisfying the following hypotheses for $\eps>0$ is  an $\eps$\emph{-small noise system}:
\begin{description}
\item[\textbf{H1}] $F_0$ is dissipative, i.e. there is a compact interval $[0,M]$ and $T\ge 1$ such that $F^T_0(x)\in [0,M]$ for all $x\in [0,\infty)$,
\item[\textbf{H2} ] $\Prob[ F_0(x)-\eps \le F(x,\xi_1) \le F_0(x)+\eps]=1$ for all $x\in \R_+$,
\item[\textbf{H3}] for all $x\in \R_+$ and all Borel sets $U\subset [(F_0(x)-\eps)^+, F_0(x)+\eps]$ with positive Lebesgue measure, there exist $\alpha>0$ and $\gamma>0$ such that 
\[
\Prob[F(z,\xi_1) \in U]>\alpha
\]
for all $z \in [(x-\gamma)^+,x+\gamma]$.
\end{description}
The first assumption ensures that the unperturbed dynamics remain uniformly bounded. The second assumption implies that the noise is $\eps$-small, while the third assumption implies the noise is locally absolutely continuous.

\begin{proposition}\label{prop:accessibility}  
Assume the difference equation $x_{t+1}=F_0(x_t)$ has no positive attractor. Then there exists   a decreasing function $\eps: \R_+ \rightarrow \R_+$ such that, for any $M>0$, there exists an invariant set $K \supset [0,M]$ such that $\{0\}$ is accessible from $K$ whenever \eqref{eq:model} is an $\eps(M)$-small noise system.
 \end{proposition}
 
As a direct consequence of Theorem \ref{thm:extinction} and Proposition \ref{prop:accessibility}, we have

\begin{corollary}\label{corollary}
For any $M>0$, there exists $\eps_0>0$ such that if \eqref{eq:model} is an $\eps$-small noise system for $\eps \le \eps_0$, the dynamics induced by $F_0$ has no positive attractor, and assumptions \textbf{A1-2} hold, then 
\[
\Prob\left[\lim_{t\to\infty} X_t =0 \vert X_0=x\right]=1,
\]
for all $x \in [0,M]$.

\end{corollary}
\subsection{Persistence}

When $\E[\log f(0,\xi_1)]>0$ and there is only negative density dependence, Theorem~\ref{thm:scalar} ensured the system is stochastically persistent. The following theorem shows that this criterion also is sufficient for models that account for negative and positive density dependence. 
\begin{theorem}\label{thm:persistence} 
If
\begin{itemize}
\item[(i)] $\E[\log f(0,\xi_1)]>0$, and
\item[(ii)] there exist $x_c>0$ such that $\E[\sup_{\{x>x_c\}}\log f(x,\xi_1)]<0$, 

\end{itemize}
then for all $\epsilon>0$ there exists $M>0$ such that 
\[
\limsup_{t\to\infty} \Pi_t( [1/M,M] ) \ge 1-\epsilon \mbox{ almost surely}
\]
whenever $X_0=x>0$.
\end{theorem}
\begin{remark}
If there exists $x_c>0$ such that $f(x,\xi)$ is a decreasing function in x on $[x_c,\infty)$,  $\E[\log f(x_c,\xi_1)]<0$, and $\E[\log^+ f(x,\xi_1) ]<\infty$ for all $x\le x_c$, then condition (ii) in Theorem \ref{thm:extinction} is satisfied.
\end{remark}

When the invasion criteria is not satisfied (i.e. $\E[\log f(0,\xi_1)]<0$), conditional persistence  may still occur. For instance, suppose the stochastic dynamics have a \emph{positive invariant set $B\subset (0,\infty)$}: there exists $\gamma>0$ such that $B\subset [\gamma,\infty)$ and $\Prob_x[X_t \in B $ for all $ t\ge0]=1$ for all $x\in B$. When such a positive invariant set exists, populations whose initial density lie in $B$ persist. The following proposition implies that conditional persistence only occurs if there is such a positive invariant set. 

\begin{proposition}\label{prop:invariant-set-compact}
Assume \textbf{A1}-\textbf{A2}. If \eqref{eq:model} is bounded in $[0,M]$ (i.e. $\Prob[X_t< M$ for all $t \ge 0]=1$), then either $\lim_{t\to\infty}X_t=0$ with probability one whenever $X_0=x\ge 0$, or there exists a positive invariant set $B \subset (0,M]$.
\end{proposition}

In the case of small noise, the following proposition implies the existence of a positive attractor for the unperturbed dynamics is sufficient for the existence of a positive invariant set. In particular,  conditional persistence is possible when $\E[\log f(0,\xi_1)]<0$.

\begin{proposition}\label{prop:invariant-set} 
Assume that $A \subset (0,\infty)$ is an attractor for the difference equation $x_{t+1}=F_0(x_t)$. Then there exists a bounded positive invariant set $K$ whenever \eqref{eq:model} satisfies \textnormal{\textbf{H2}} for $\eps>0$ sufficiently small.
\end{proposition}

\subsection{Mate-limitation and predator-saturation with negative density-dependence}
To illustrate Theorems~\ref{thm:extinction},\ref{thm:persistence} and Propositions~\ref{prop:accessibility},\ref{prop:invariant-set}  we apply them to models accounting for negative density-dependence and positive density-dependence via mate-limitation or predator-saturation. The deterministic version of these models were analyzed by~\citet{tpb-03}. 

To account for negative density-dependence, we use a Ricker type equation. In the case of the  mate-limitation model, the fitness function becomes
\begin{equation}\label{eq:MLNDD}
f(x,\xi)= \exp(r-a\,x)\frac{x}{h+x} \mbox{ where }\xi =(r,a,h)
\end{equation}
where $r$ is the intrinsic rate of growth in the absence of mate-limitation, $a$ measures the strength of infraspecific competition, and $h$ is the half-saturation constant as described in Section 3.2. In the absence of stochastic variation in the parameters $r,a,h$, the dynamics of persistence and extinction come in three types~\citep{tpb-03}. If $f(x,\xi)<1$ for all $x\ge 0$, then all initial conditions go asymptotically to extinction. If $f(x,\xi)>1$ for some $x>0$, then dynamics of extinction are governed by the smallest positive fixed point $M$ and the critical point $C$ of $F(x)=xf(x,\xi)$. If $F(F(C))>M$, then there is a positive attractor in the interval $(M,\infty)$ for the deterministic dynamics. Alternatively, if $F(F(C))<M$, then the model exhibits  essential extinction: asymptotic extinction occurs for Lebesgue almost every initial density, but there is an infinite number of unstable positive periodic orbits. In particular, there is no positive attractor.

\begin{figure}
\includegraphics[width=0.8\textwidth]{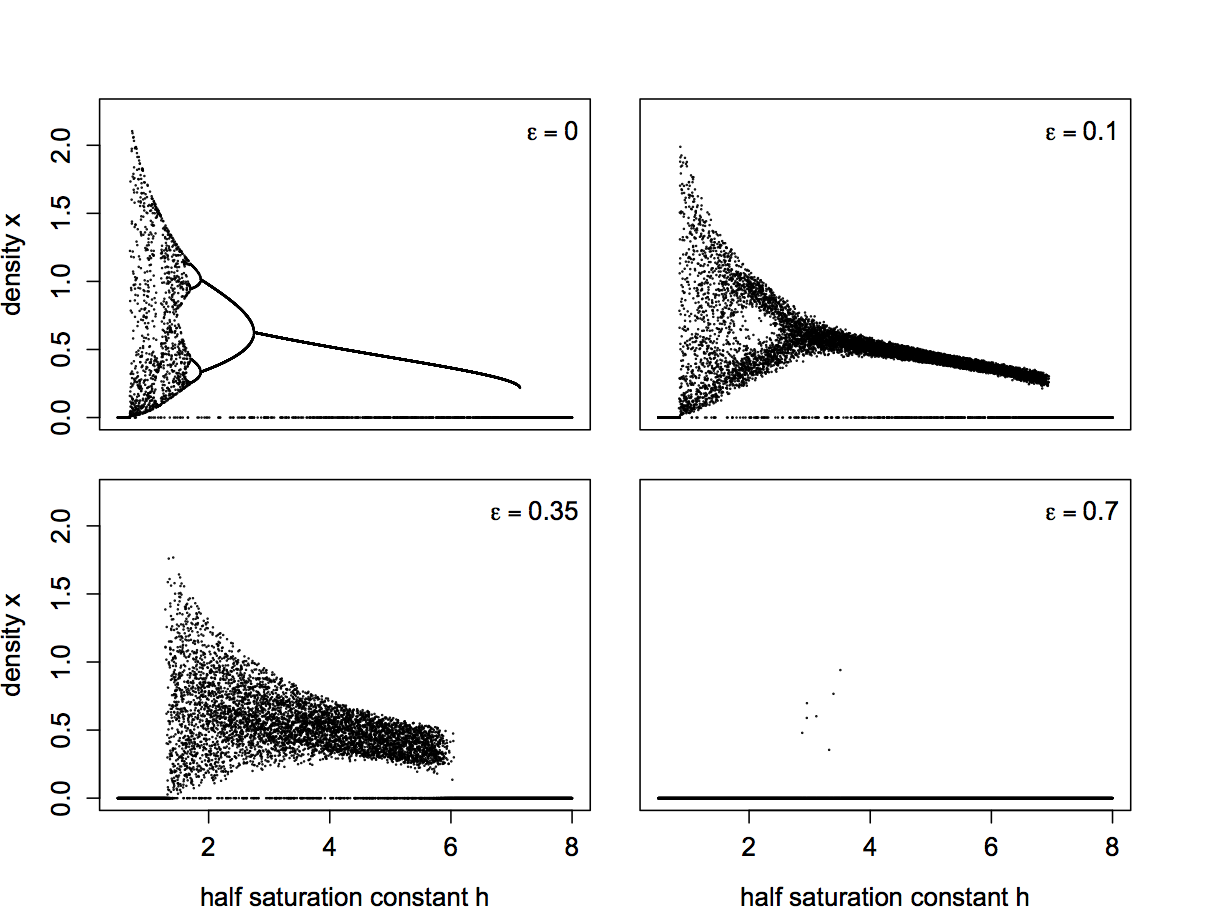}
\caption{Asymptotic dynamics of extinction and persistence for the stochastic mate limitation model with negative density-dependence. For each parameter value, the model was simulated $10,000$ time steps for multiple initial conditions. The final $1,000$ points of each simulation are plotted. Model details: The fitness function is $f(x,\xi)= \exp(r-ax)\frac{ x}{h+x}$ where $r$ is uniformly distributed on $[4.5-\epsilon,4.5+\epsilon]$  }\label{fig:PDD-NDD-ML}
\end{figure}

To account for stochasticity, we assume, for illustrative purposes, that $r_t$ is uniformly distributed on the interval $[r-\epsilon,r+\epsilon]$ with $r>0$ and $0<\epsilon<r$. Furthermore, we assume that $a=1$ and $h>0$.  As $\E[\log f(0,\xi_t)]=-\infty$, Theorem~\ref{thm:extinction} implies that $\lim_{t\to\infty}X_t= 0$ with positive probability for initial conditions $X_0$ sufficiently close to $0$. When the deterministic dynamics support a positive attractor (i.e. $F(F(C))>M$) and the noise is sufficiently small (i..e $\epsilon>0$ sufficiently small), Proposition~\ref{prop:invariant-set} implies that the density $X_t$ for the stochastic model remains in a positive compact interval contained in $(M,\infty)$. Alternatively, if the deterministic dynamics exhibit essential extinction and the noise is sufficiently small, Proposition~\ref{prop:accessibility} implies $\lim_{t\to\infty} X_t =0$ with probability one for all initial densities. Finally, when $\epsilon$ is sufficiently close to $r$ (i.e. the noise is sufficiently large), Theorem~\ref{thm:extinction} implies that $\lim_{t\to\infty}X_t=0$ with probability one for all positive initial conditions. This later outcomes occurs whether or not the deterministic dynamics support a positive attractor. All of these outcomes are illustrated in Fig.~\ref{fig:PDD-NDD-ML}.

For the predator-saturation model, we use the fitness function   
\begin{equation}\label{eq:MLNDD}
f(x,\xi)= \exp\left(r-a\,x-\frac{P}{h+x}\right) \mbox{ where }\xi =(r,a,h,P)
\end{equation}
where $h$ and $P$ are the half-saturation constant and the maximal predation rate, respectively, as described in Section 3.3. The dynamics of persistence and extinction for this model without stochastic variation come in four types~\citep{tpb-03}. If $f(0,\xi)>1$, then there is a positive attractor whose basin contains all positive initial densities. If $f(x,\xi)<1$ for all $x\ge 0$, then all initial conditions go asymptotically to extinction. If $f(x,\xi)>1$ for some $x>0$, then dynamics of extinction are governed by the smallest positive fixed point $M$ and the critical point $C$ of $F(x)=xf(x,\xi)$. If $F(F(C))>M$, then there is a positive attractor in the interval $(M,\infty)$ for the deterministic dynamics. Alternatively, if $F(F(C))<M$, then the model exhibits  essential extinction.

To account for stochasticity, we assume for simplicity that $P_t$ is uniformly distributed on the interval $[P(1-\epsilon),P(1+\epsilon)]$ for some $P>0$ and $0<\epsilon<1$. Furthermore, we assume that $a=1$, $r>0$, and $h>0$.  When $\E[\log f(0,\xi_t)]=r-P>0$, Theorem~\ref{thm:persistence} implies the system is stochastically persistent. Alternatively, when $\E[\log f(0,\xi_t)]=r-P<0$, Theorem~\ref{thm:extinction} implies that $\lim_{t\to\infty}X_t= 0$ with positive probability for initial conditions $X_0$ sufficiently close to $0$. Assume  $r<P$. If  the deterministic dynamics support a positive attractor (i.e. $F(F(C))>M$) and the noise is sufficiently small (i.e. $\epsilon>0$ sufficiently small), Proposition~\ref{prop:invariant-set} implies that the density $X_t$ for the stochastic model remains in a positive compact interval contained in $(M,\infty)$. Hence, the population exhibits conditional persistence. Alternatively, if the deterministic dynamics exhibit essential extinction and the noise is sufficiently small, Proposition~\ref{prop:accessibility} implies $\lim_{t\to\infty} X_t =0$ with probability one for all initial densities. Finally, when $\epsilon$ is sufficiently close to $1$ (i.e. the noise is sufficiently large) and $P>r$, Theorem~\ref{thm:extinction} implies that $\lim_{t\to\infty}X_t=0$ with probability one for all positive initial conditions. All of these outcomes  are illustrated in Fig.~\ref{fig:PDD-NDD-Predation}.

\begin{figure}
\includegraphics[width=0.8\textwidth]{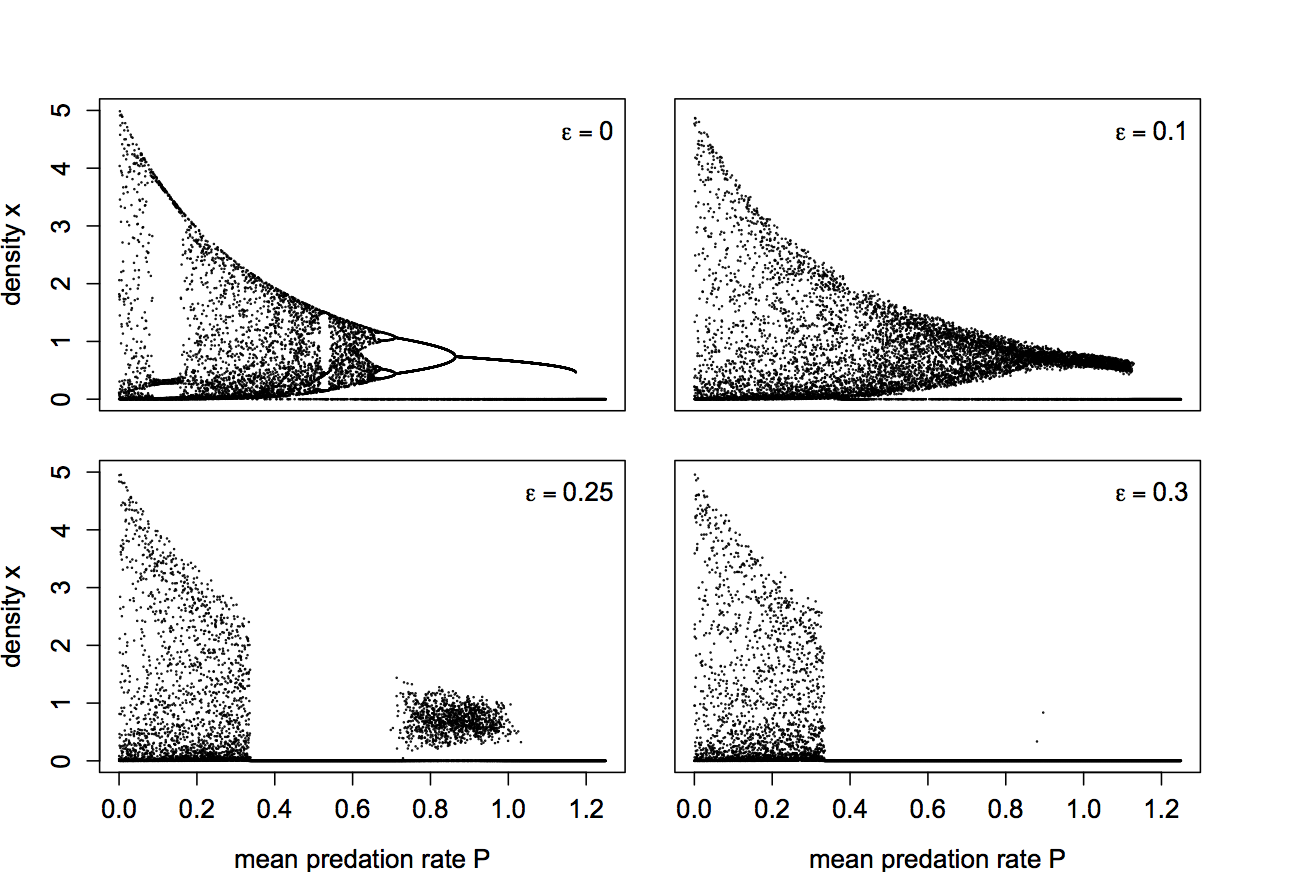}
\caption{Asymptotic dynamics of extinction and persistence for the stochastic predator saturation model with negative density-dependence.   For each parameter value, the model was simulated $10,000$ time steps for multiple initial conditions. The final $1,000$ points of each simulation are plotted. Model details:  fitness function $f(x,\xi)= \exp(4-4x-P_t/(1/12+x))$ with $P$ uniformly distributed on $\bar{P}[1-\epsilon,1+\epsilon]$. }\label{fig:PDD-NDD-Predation}
\end{figure}

\section{Discussion}

A demographic Allee effect occurs when individual fitness, at low densities, increases with population density. If individuals on average replace themselves at very low densities, then the population exhibits a weak Allee effect. Alternatively, if there is a critical density below which individuals do not replace themselves and above which where they do, then the population exhibits a strong Allee effect. It is frequently argued that environmental stochasticity coupled with a strong Allee effect can increase the likelihood population fall below the critical threshold, rendering them particularly vulnerable to extinction~\citep{courchamp-etal-99,stephens-etal-99}. While this conclusion is supported, in part, by mathematical and numerical analyses of stochastic differential equation models~\citep{dennis-02,krstic-jovanovic-10,yang-jiang-11}, these analyses are specific to a modified Logistic growth model with Brownian fluctuations in the log population densities. Here, we analyzed discrete-time models allowing for 
general forms of density-dependent feedbacks and randomly fluctuating vital rates. Our analysis demonstrates that environmental stochasticity can convert weak Allee effects to strong Allee effects and that the risk of asymptotic extinction with strong Allee effects depends on the interaction between density-dependent feedbacks and environmental stochasticity. 

When environmental fluctuations ($\xi_t$)  drive population dynamics ($x_{t+1}=f(x_t,\xi_{t+1})x_t$), an Allee effect is best defined in terms of  the geometric mean $G(x)=\exp(\E[\log f(x,\xi_t)])$ of fitness. If the geometric mean $G(x)$ is an increasing function at low densities, an Allee effect occurs. If this geometric mean is greater than one at low densities ($G(0)>1$), then we proved that the Allee effect is weak in that the population stochastically persists: the population densities spends arbitrarily little time at arbitrarily low densities. When the geometric mean is less than one at low densities ($G(0)<1$), the stochastic Allee  effect is strong: for populations starting at sufficiently low densities,  the population density asymptotically approaches zero with positive probability. Since the geometric mean $G(0)$ in general does not equal the intrinsic fitness $f(0,\E[\xi_t])$ at the average environmental condition, environmental stochasticity can, in and of itself, shift weak Allee effects to strong Allee effects and vice versa. For example, a shift from a weak Allee effect to a strong Allee effect can occur when a population's predator has a fluctuating half-saturation constant. Specifically, for the predator saturation model considered here, the geometric mean at low densities equals $G(0)=r- \E[P/h_t]$ where $r$ is the intrinsic rate of growth of the focal population, $P$ is proportional to the predator density, and $h_t$ is the fluctuating half saturation constant of the predator. As Jensen's inequality implies that  $G(0)< r - P/ \E[h_t]$, fluctuations in $h_t$ can decrease the value of $G(0)$ from $>1$ to $<1$ and thereby shift a weak Allee effect to a strong Allee effect.  

In the absence of negative density-dependent feedbacks, there is a dynamical trichotomy: asymptotic extinction for all initial densities, unbounded population growth for all positive initial conditions, or a strong Allee effect (i.e. $G(0)<1$ but $G(x)<1$ for sufficiently large $x$). When a strong Allee effect occurs and environmental fluctuations are large (i.e. the support of $\log f(x,\xi_t)$ is the entire real line for all $x>0$), populations either go asymptotically to extinction or grow without bound with probability one. Moreover, both outcomes occur with positive probability for all positive initial conditions. 

\citet{liebhold-bascompte-03}  used models with only positive density-dependence to examine numerically  the joint effects of Allee effects, environmental stochasticity, and externally imposed mortality on the probability of successfully exterminating an invasive species. Their fitness function  was
\[
f(x,\xi_t)=\exp\left(\gamma (x-C) + \xi_t\right)
\]
were $C$ is the deterministic Allee threshold, $\gamma$ is the ``intrinsic rate of natural increase,'' and $\xi_t$ are normal random variables with mean $0$. Since $G(0)=\exp(-\gamma C)<1$ and $\lim_{x\to\infty} G(x)=+\infty$ for this model, our results imply both extinction and unbounded growth occur with positive probability and, thereby, provide a rigorous mathematical foundation for \citet{liebhold-bascompte-03}'s numerical analysis. Consistent with our simulations of a stochastic mate limitation model, \citet{liebhold-bascompte-03}'s found that the probability of persistence increases in a sigmoidal fashion with initial population density. In particular, environmental stochasticity increases the probability of persistence for populations initiated at low densities by pushing their densities above the deterministic Allee threshold. Conversely, for populations initiated at higher densities, environmental stochasticity can increase the risk of asymptotic extinction by pushing densities below this threshold. Indeed, we proved that the probability of asymptotic extinction approaches zero as initial population densities get large and the probability of asymptotic extinction approaches one as initial population densities get small. 

Since populations do not grow without bound, negative density-dependent feedbacks ultimately dominate population growth at higher population densities~\citep{wolda-dennis-93, turchin-95, harrison-cappuccino-95}. While stochastic persistence never occurs with a strong Allee effect,  extinction need not occur with probability one. Whether or not extinction occurs for all positive initial densities with probability one depends on a delicate interplay between the nonlinearities of the model and the form of environmental stochasticity.  A sufficient condition for unconditional extinction (i.e. extinction with probability one for all initial conditions) is that the extinction set $\{0\}$ is ``attainable'' from every population density state. Attainability roughly means that the population densities become arbitrarily small at some point in time with probability one. For bounded population state spaces, we proved a dichotomy: either there exists a positive invariant set for the process or $\{0\}$ is attainable in which case there is unconditional extinction. Whether this dichotomy extends to unbounded population state spaces remains an open problem. 

When environmental stochasticity is weak and there is a strong Allee effect,  the ``unperturbed'' population dynamics determines whether extinction occurs for all initial conditions or not. By ``weak'' we mean  that the unperturbed dynamics $F$ are subject to small, compactly supported random perturbations (i.e. $x_{t+1}-F(x_t)$ lies in an interval $[-\varepsilon,\varepsilon]$ for $\varepsilon>0$ small). The existence of a  positive attractor is necessary for conditional persistence in the face of weak environmental stochasticity. This result confirms the consensus in the mathematical biology community, that the existence of a positive attractor ensures that population trajectories can remain bounded away form extinction in the presence of small perturbations \citep{jtb-06}. 

For populations exhibiting a strong Allee effect and conditional persistence at low levels of environmental stochasticity, there is always a critical level of environmental stochasticity above which asymptotic extinction occurs with probability one for all initial population densities. Mathematically, there is a transition from the extinction set $\{0\}$ being inaccessible for part of the population state space at low levels of environmental stochasticity to $\{0\}$ being accessible for the entire population state space at higher levels of environmental stochasticity. We have illustrated this transition in stochastic models of mate-limitation and predator-saturation with negative density-dependence. Surprisingly, for the predator-saturation models, our numerical results show that environmental stochasticity can lead to asymptotic extinction at intermediate predation rates despite conditional persistence occurring at higher and lower predation rates. This effect, most likely, is due to the opposing effects of predation on overcompensatory feedbacks and the Allee threshold resulting in a larger basin of attraction for the extinction state at intermediate predation
 rates. 

While our analysis provides some initial insights into the interactive effects of Allee effects and environmental stochasticity on asymptotic extinction risk, many challenges remain. Many populations exhibit spatial, ontogenetic, social, or genetic structure. Proving multivariate analogs to the results proven here could provide insights on how population structure interacts with the effects considered here to determine population persistence or extinction. Furthermore, all populations consist of a finite number of individuals whose fates are partially uncorrelated. Hence, they experience demographic as well as environmental stochasticity. In accounting for bounded, finite population sizes in stochastic models, extinction in finite time is inevitable. However, these models often exhibit meta-stable behavior in which the populations persist for long periods of time despite both forms of stochasticity and Allee effects. This meta-stable behavior often is associated with quasi-stationary distributions of the finite-state models. Studying to what extent these distributions have well definite limits in an ``infinite-population size'' limit is likely to provide insights into these metastable behaviors~\citep{aap-14} and provide a more rigorous framework to evaluate the joint effects of stochasticity and Allee effects on population persistence and ultimately their consequences for conservation and management.

\section{Appendix}
This Appendix provides proofs of all the results in the main text. In Section \ref{accessibility}, we prove a general convergence result based on an accessibility assumption that leads to the proofs of Theorems~\ref{thm:extinction} in Section \ref{proof:extinction} and Theorem \ref{thm:PDD} in Section \ref{proof:PDD}. In Section \ref{proof:prop-accessibility}, we prove Proposition \ref{prop:accessibility}. In Section \ref{proof:persistence}, we prove Proposition \ref{prop:invariant-set-compact}, and in Section \ref{proof:invariant}, we prove Proposition \ref{prop:invariant-set}.

We begin with some useful definitions and notations. Let $\mathcal{B}$ is the Borel $\sigma$-algebra on $\R_+$. Let $\delta_y$ denote a Dirac measure at $y$, i.e. $\delta_{y}(A) =1$ if $y\in A$ and $0$ otherwise for any set $A \in \mathcal{B}$. Let $\R_+ \cup \{\infty\}$ be the one-point compactification of $\R_+$ and assume that $\{\infty\}$ is a fixed point for the system \eqref{eq:model}. For a sequence $(x_t)_{t\ge0}\subset \R_+ \cup \{\infty\}$, we write $x_t \xrightarrow[t \to \infty]{} D$ when $(x_t)_{t\ge0}$ converges to $D\subset \R_+ \cup \{\infty\}$, i.e. if for any neighborhood $U$ of $D$, there exists $T>0$ such that $x_t\in U$ for all $t\ge T$. 

We consider the trajectory space formed by the product $\Omega = \R_+^\N$ equipped with the product $\sigma$-algebra $\mathcal{B}^\N$. For any $x\in \R_+$ (viewed as an initial condition of trajectory), there exists a probability measure $\Prob_x$ on $\Omega$ satisfying
\[
\Prob_x[\{\omega \in \Omega : \omega_0 \in A_0,\dots,\omega_{k}\in A_k\}] =\Prob[X_0\in A_0,\dots,X_k \in A_k \vert X_0 =x]
\]
for any Borel sets $A_0,\dots,A_k \subset \R_+$, and $\Prob_x[\{\omega \in \Omega : \omega_0 =x\}]= 1$. The random variables $X_t$ are the projection maps
\[
\begin{array}{cccccc}
  X_t:&\Omega& \rightarrow & \R_+   \\
  &  \omega  &\mapsto &\omega_t.
\end{array}
\]
For the proof of Theorems~\ref{thm:extinction} and \ref{thm:PDD}, we  consider  the space $E^\N$ of the environmental trajectories equipped with the product $\sigma$-algebra $\mathcal{E}^\N$, and  the probability measure $\ProbQ$ on $E^\N$ satisfying
\[
\ProbQ(\{\bbe \in E^\N : e_0 \in E_0,\dots,e_{k}\in E_k\}) =\Prob(\xi_0\in E_0,\dots,\xi_k \in E_k)
\]
for any Borel sets $E_0,\dots,E_k \subset E$. For now on, when we write $\bbe \in E^\N$, we mean $\bbe = (e_t)_{t\ge0}$. Since $E$ is a Polish space (i.e. separable completely metrizable topological space), the space $E^\N$ endowed with the product topology is Polish as well. Therefore, by the Kolmogorov consistency theorem, the probability measure $\ProbQ$ is well defined. In this setting, the random variable $\xi_t$ is the projection map
\[
\begin{array}{cccccc}
  \xi_t:&E^\N& \rightarrow & \R_+   \\
  &  \bbe  &\mapsto &e_t.
\end{array}
\]
We use the common notation $\E$ (resp. $\E_x$) for the expectation with respect to the probability measure $\ProbQ$ (resp. $\Prob_x$).

Let $x\in \R_+$ and $\Omega_x=\{\omega \in \Omega \ : \ \omega_0=x\}$ be the cylinder of the trajectories starting from $x$. The continuous function $\vphi: E^\N \rightarrow \Omega_x$ defined component-wise by 
\begin{equation}\label{eq-phi}
\vphi(\bbe)_t = X_t(\vphi(\bbe))=xf(x,e_1)f(X_1(\vphi(\bbe)),e_2)\cdots f(X_{t-1}(\vphi(\bbe)),e_t),
\end{equation}
links the probability measures  $\ProbQ$ and $\Prob_x$. In fact, the pushforward measure of $\ProbQ$ by $\vphi$ is the probability measure $\Prob_x$, i.e. for any Borel set $A\subset \Omega_x$, $\Prob_x(A)=\ProbQ(\vphi^{-1}(A))$. 

Recall that a set $A \subset \R_+ \cup \{\infty\}$ is \emph{accessible} from $B\subset  \R_+ \cup  \{\infty\}$ if for any neighborhood $U$ of $A$, there exists $\gamma>0$ such that
\[
\Prob_x[\exists t \ge 0 \ : \ X_t \in U ]>\gamma
\] 
for all $x\in B$. A subset $C\subset [0,\infty)$ is \emph{invariant} for the system \eqref{eq:model} if $\Prob_x[ X_{1} \in C]=1$ for all $x\in C$, and it is \emph{positive} if $C\subset (0,\infty)$.


\subsection{Convergence result}\label{accessibility}

\begin{proposition}\label{general-global-result} Let $B\subset \R_+$ be an invariant subset for the system \eqref{eq:model}, and $A \subset B$ be an accessible set from $B$. Assume that there exists $0<\delta<1$ and a neighborhood $U$ of $A$ such that
\[
\Prob_x\left[ X_t \xrightarrow[t \to \infty]{} A \right]\ge1-\delta,
\]
for all $x \in U$. Then
\[
\Prob_x\left[X_t \xrightarrow[t \to \infty]{} A \right]=1
\]
for all $x \in B$.
\end{proposition}
\begin{proof}
Define the event $\mathcal{C}=\{X_t \xrightarrow[t \to \infty]{} A \}$. By assumption there exists $\delta>0$ and a neighborhood $U$ of $A$ such that
\[
\Prob_x\left[ \mathcal{C}\right] \ge 1-\delta,
\]
for all $x\in U$. Fix $x\in B$ and define the stopping time $\tau = \inf \{t\ge 0 \ : X_t \in U\}$. Since $A$ is accessible from $B$, there exists $\gamma >0$ such that $\Prob_x[\tau<\infty]>\gamma$. The strong Markov property implies that
\begin{eqnarray*}
\Prob_x\left[ \mathcal{C}\right] &\ge& \int_\Omega \Prob_{X_\tau}\left[ \mathcal{C}\right] \mathds{1}_{\{ \tau<\infty\}} d\Prob_x\\
&\ge&(1-\delta)\gamma,
\end{eqnarray*}
The L\'{e}vy zero-one law implies that $\lim_{t\to \infty} \E_x\left[ \mathds{1}_{\mathcal{C}} | \mathcal{F}_t\right]= \mathds{1}_{\mathcal{C}}$ $\Prob_x$-almost surely, where $\mathcal{F}_t$ is the $\sigma$-algebra generated by $\{X_1,\dots,X_t\}$. On the other hand, the Markov property implies that $\E_x\left[ \mathds{1}_{\mathcal{C}} | \mathcal{F}_t\right]= \Prob_{X_t}[\mathcal{C}] \ge (1-\delta)\gamma$. Hence $\Prob_x[\mathcal{C}]=1$. 

\end{proof}

\subsection{Proof of Theorem~\ref{thm:extinction} }\label{proof:extinction}

To prove the local extinction result, assume $\E[\log f(0,\xi_t)]<0$ and that there exists $\gamma>0$ such that $x\mapsto f(x,\xi)$ is increasing on $[0,\gamma)$ for all $\xi \in E$. Since $\E[\log f(0,\xi_1)]<0$ and $x \mapsto f(x,\xi)$ is monotone on $[0,\gamma]$, there exists $0<x^*<\gamma$ such that $\E[\log f(x^*,\xi_1)]<0$ and $f(x,\xi) \le f(x^*,\xi)$ for all $x\in [0,x^*)$ and all $\xi \in E$. The Law of Large Numbers implies that 
\[
 \lim_{t\to\infty} \frac{1}{t} \sum_{s=0}^{t-1} \log f(x^* ,\xi_s) <0,
\]
$\ProbQ$-almost surely.
Define the random variable
\[
R = \inf \frac{1}{f(x_0,\xi_1) \dots f(x_{t-1}, \xi_t)},
\]
where the infimum is taken over the set $\bigcup_{t\ge 1}\{(x_0,\dots,x_{t-1}) \in [0,x^*)^t\}$.
As
\begin{equation}\label{eq-LLN}
\limsup_{t\to\infty} \frac{1}{t} \sum_{s=0}^{t-1} \log f(x_s ,\xi_{s+1})\le \lim_{t\to\infty} \frac{1}{t} \sum_{s=0}^{t-1} \log f(x^* ,\xi_s)<0
\end{equation}
$\ProbQ$-almost surely for any sequence $\{x_t\}_{t\ge0}$ lying in $[0,x^*)$,  $R>0$ $\ProbQ$-almost surely. Let $\Gamma \subset E^\N$ be the set of probability $1$ for which both the limit in \eqref{eq-LLN} exists and $R>0$. Choose $\bbe \in  \Gamma$ and $x\in [0,x^* R(\bbe)]$. Let $\vphi: E^\N \rightarrow \Omega_x$ be the function defined by \eqref{eq-phi}. The definition of $R(\bbe)$ implies by induction that  $X_t(\vphi(\bbe)) \le \delta $ for all $t\ge 1$. Hence, our choice of $x^*$ implies that 
\begin{eqnarray*}
\limsup_{t\to\infty}\frac{1}{t} \log X_t(\vphi(\bbe)) &=& \lim_{t\to \infty} \frac{1}{t} \sum_{s=0}^{t-1} \log (f(X_s(\vphi(\bbe)),e_{s+1})\\
&\le&\lim_{t\to \infty} \frac{1}{t} \sum_{s=0}^{t-1} \log (f(x^*,e_{s+1})\\
&<& 0
\end{eqnarray*}
for all $\bbe \in \Gamma$. As $\ProbQ[\Gamma]=1$, 
\[
\Prob_x[\limsup_{t\to\infty}\frac{1}{t} \log X_t=0 ]= \ProbQ[\limsup_{t\to\infty}\frac{1}{t} \log X_t\circ \vphi=0] \ge \ProbQ[ R\ge\frac{1}{x^* n}],
\]
for all $n >0$ and $x \in (0,\frac{1}{n}]$.

Fix $\delta>0$. Since $\{R \ge \frac{1}{x^* n}\}_{n\ge 1}$ is an increasing sequence of events and $\ProbQ[\cup_n \{R \ge \frac{1}{x^*n}\} ]=1$, $\lim_{n\to \infty}\ProbQ[R \ge \frac{1}{x^* n}]=1$ which implies that there exists $N>0$ such that $\ProbQ[ R \ge \frac{1}{x^* N}]\ge 1-\delta$. Hence
\begin{eqnarray*}
\Prob_x\left[\lim_{t\to\infty} X_t =0 \right] &\ge&  \ProbQ\left[R \ge \frac{1}{x^*N} \right]\\
&\ge& 1-\delta,
\end{eqnarray*}
for all $x \in (0,\frac{1}{N}]$. 

Now assume  $\{0\}$ is accessible from $[0,M)$ for some $M > 0$ (possibly $+\infty$). Applying Proposition~\ref{general-global-result} to $A=\{0\}$ and $B= [0,M]$ (resp. $B= [0,\infty)$) implies $\lim_{t\to\infty}X_t =0$ with probability one whenever $X_0=x\in [0,M)$.

\subsection{Proof of Theorem~\ref{thm:PDD}}\label{proof:PDD}

To prove the extinction result, suppose that $\gamma=\E[\log f_\infty(\xi_t)]<0$ and fix $x\in R_+$. Let Let $\vphi: E^\N \rightarrow \Omega_x$ be the function defined by \eqref{eq-phi}. Since $ x\mapsto f(x,\xi)$ is an increasing function for all $\xi \in E$, we have
\[
X_t(\vphi(\bbe)) = \prod_{s=0}^{t-1} f(X_{s}(\vphi(\bbe)),e_{s+1})x < \prod_{s=1}^{t} f_\infty(e_{s})x,
\]
for all $\bbe \in E^\N $.
By the Law of Large Numbers,
\[
\limsup_{t\to\infty} \frac{1}{t}\log X_t(\vphi(\bbe)) \le \lim_{t\to\infty} \frac{1}{t}\left(\sum_{s=1}^{t}\log f_\infty(e_{s}) +\log x\right) = \gamma<0
\]
for $\ProbQ$-almost all $\bbe \in E^\N$. Therefore
\[
\Prob_x[\lim_{t\to\infty} X_t=0]=\ProbQ[\lim_{t\to\infty} X_t\circ \vphi=0]=1,
\]
which completes the proof of the first assertion.  

To prove the unbounded growth result, suppose that $\alpha=\E[\log f(0,\xi_t)]>0$ and fix $x\in (0,\infty)$. Since $x \mapsto f(x,\xi)$ is an increasing function for all $\xi \in E$, we have
\[
X_t(\vphi(\bbe)) = \prod_{s=0}^{t-1} f(X_{s}(\vphi(\bbe)),e_{s+1})x  >\prod_{s=1}^{t} f(0,e_{s})x
\]
for all $\bbe\in E^\N$.
By the Law of Large Numbers,
\[
\liminf_{t\to\infty} \frac{1}{t}\log X_t(\vphi(\bbe)) \ge \lim_{t\to\infty} \frac{1}{t}\left(\sum_{s=1}^{t}\log f(0,e_{s}) +\log x\right) = \alpha>0
\]
$\ProbQ$-almost all $\bbe \in E^\N$. Therefore
\[
\Prob_x[\lim_{t\to\infty} X_t=\infty]=\ProbQ[\lim_{t\to\infty}X_t\circ \vphi=\infty]=1,
\]
which completes the proof of the first assertion.  

To prove the Allee effect result, fix $0<\delta<1$ and assume that $\E[\log f(0,\xi_1)]<0$ and $\E[\log f_\infty(\xi_1)]>0$. The first assertion of Theorem \ref{thm:extinction} implies that there exists $m>0$ such that 
\begin{equation}\label{eq:zero}
\Prob_x\left[ \lim_{t\to\infty} X_t =0 \right]\ge1-\delta,
\end{equation}
for all $x\in(0,m]$. To prove the second part of the result, consider the process $Y_t=\frac{1}{X_t}$ conditioned to the event $\{X_0>0\}$. It satisfies the following stochastic difference equation:
\[
Y_{t+1}=g(Y_t,\xi_{t+1}),
\]
where $g: \R_+ \times \Omega \rightarrow \R$ is the continuous function defined by $g(y,\xi)=\frac{1}{f(\frac{1}{y}, \xi)}$ for all $y>0$ and $g(0,\xi)=\frac{1}{f_\infty(\xi)}$. By definition of $Y_t$, we have, for all $m>0$,
\[
\Prob_y\left[ \lim_{t\to\infty} Y_t =0 \right] =\Prob_{\frac{1}{y}}\left[ \lim_{t\to\infty} X_t =\infty \right]
\]
for all $y\in (0,m]$.

Since, $\E[\log g(0,\xi_1)]=-\E[\log f_\infty(\xi_1)]<0$, the first assertion of Theorem \ref{thm:extinction} applied to $(Y_t)_{t\ge0}$ implies that for any $0<\delta<1$, there exist $L>0$ such that 
\begin{equation*}\label{eq:infty1}
\Prob_y\left[  \lim_{t\to\infty} Y_t =0 \right] \ge 1- \delta, 
\end{equation*}
for all $y\in[0,L)$. Therefore, for any $0<\delta<1$, there exists $M=\frac{1}{L}>0$ such that
\begin{equation}\label{eq:infty}
\Prob_x\left[  \lim_{t\to\infty} X_t = \infty \right]\ge1-\delta,
\end{equation}
for all $x \in  [M,\infty)$. 

To prove the last part, assume that $\{0,\infty\}$ is accessible from $\R_+$ and fix $\delta>0$. By \eqref{eq:zero} and \eqref{eq:infty}, there exist $m,M>0$ such that 
\[
\Prob_x\left[ X_t \xrightarrow[t \to \infty]{} \{0,\infty\} \right]\ge1-\delta,
\]
for all $x \in [0,m] \cup [M,\infty)$. Proposition \ref{general-global-result} applied to $A=\{0,\infty\}$, $B=\R_+$ and $U=[0,m] \cup [M,\infty)$ concludes the proof.

\subsection{Proof of Proposition \ref{prop:accessibility}}\label{proof:prop-accessibility}

The proof consists of combining two deterministic arguments with a probabilistic argument. The three of them use the concept of $(\eps,T)$-\emph{chain} introduced by \cite{conley-78}. An $(\eps,T)$-chain from $x$ to $y$ in $\R$, for a mapping $F_0:\R_+ \rightarrow \R_+$, is a sequence of points $x_0=x,x_1,\dots,x_{T-1}=y$ in $\R_+$ such that for any $s=0,\dots,T-2$,  $\vert x_{i+1} -F_0(x_{i})\vert < \eps$. $x$ \emph{chains to } $y$ if for any $\eps>0$ and $T\ge 2$ there exists an $(\eps,T)$-chain from $x$ to $y$.

The following Propositions are the deterministic ingredients of the proof and are proved in \cite{jtb-06}.

\begin{proposition}\label{proposition-attractor}
Let $A$ be an attractor with basin of attraction $\mathcal{B}(A)$ and $V\subset U$ be neighborhoods of $A$ such that the closure $\overline{U}$ of $U$ is compact and contained in $\mathcal{B}(A)$. Then there exists $T\ge0$ and $\delta>0$ such that every $\delta$ chain of length $t\ge T$ starting in $U$ ends in $V$.
\end{proposition}

\begin{proposition}\label{proposition-chain}
If $F_0:\R_+ \rightarrow \R_+$ satisfies \textnormal{\textbf{H1}} and has no positive attractor, then for all $x\in \R_+$, $\eps>0$ and $T>0$ there exists an $\eps$ chain from $x$ to $0$ of length at least $T$.
\end{proposition}

The probabilistic ingredient is an adaptation of Proposition 3 in \cite{dcds-07} to our framework.
\begin{proposition}\label{proposition-att}
Assume \eqref{eq:model} satisfies \textnormal{\textbf{H3}} for $\eps_0>0$. If $x\in \R_+$ chains to $0$, then, for all $\eps\le \eps_0$, there exists a neighborhood $U_\eps$ of $x$ and $\beta_\eps >0$ such that
\[
\Prob_z[\exists t \ge 0 \ : \ X_t < \eps]>\beta_\eps
\] 
for all $z \in U_\eps$.
\end{proposition}

\begin{proof}
For any $a,b \in R_+$, define $I(a,b):= [(a-b)^+, a+b]$. Let $\eps\le \eps_0$ and $x_0=x,x_1,\dots,x_t=0$ be an $\frac{\eps}{2}$-chain from $x$ to $0$. There exists $\gamma_{t}>0$ such that $I_t:= I(F_0(x_{t-1}),\gamma_{t})\subset [0,\eps]$. Assumption \textbf{H3} implies that there exist $\gamma_{t-1}>0$ and $\alpha_t>0$ such that $I_{t-1}:= I(x_{t-1},\gamma_{t-1})\subset I(F_0(x_{t-2}),\eps)$ and 
\[
\Prob_z[X_1\in I_t]>\alpha_t
\]
for all $z \in I_{t-1}$. Since $I_{t-1}\subset I(F_0(x_{t-2}),\eps)$, assumption \textbf{H3} implies that there exist $\gamma_{t-2}>0$ and $\alpha_{t-1}>0$ such that
\[
\Prob_z[X_1\in I_{t-1}]>\alpha_{t-1}
\]
for all $z \in I_{t-2}:= I(x_{t-2},\gamma_{t-2})$. Repeating this argument, there exist $I_{0},\dots,I_t \subset \R_+$ and $\alpha_1,\dots,\alpha_t >0$ such that for all $s=1,\dots,t$
\[
\Prob_z[X_1\in I_{s}]>\alpha_{s}
\]
for all $z \in I_{s-1}$. Define $\alpha:=\min_s\alpha_s$. The Markov property implies that, for all $z \in \R_+$,
\[
\Prob_z[X_t \in I_t]=\E_z\Big[ \E_{X_1}\Big[\dots \E_{X_{t-2}}\Big[ \Prob_{X_{t-1}}[X_1 \in I_t]\Big]\Big]\Big]
\]
Since for all $s=1,\dots,t$, $\Prob_z[X_1\in I_s] >\alpha \mathds{1}_{I_{s-1}}$, 
\[
\Prob_z[X_t \in I_t ]>\alpha^t \ \text{ for all } z\in I_0.
\]
Choosing $U=I_0$ and $\beta=\alpha^t$ competes the proof of the proposition.
\end{proof}

\begin{lemma}\label{lemma:invariant}
Let $\eps>0$ and $V\subset U$ be bounded subsets of $\R_+$. Assume that the system \eqref{eq:model} satisfies \textnormal{\textbf{H2}} for $\eps$ and that there exists $T>0$ such that every $\eps$ chain of length $t\ge T$ starting in $U$ ends in $V$. Then there exists a bounded invariant set $K \supset U$ for the system \eqref{eq:model}. Moreover, if $U\subset (0,\infty)$, then $K\subset (0,\infty)$. 
\end{lemma}

\begin{proof}
Assume that $U\subset (0,\infty)$. Let $L=\sup_{x\in U}\{ F_0(x)\} +\sup_{x\in U}\{x\} + T\eps$. Assumption \textnormal{\textbf{H2}} implies that, for any $x\in U$, $X_t\in (0,L)$ for all $t\ge 0$ with probability one whenever $X_0=x$. Define the positive bounded Borel set 
\[
K:=\{x\in(0,L) \ : \ \Prob_x[\exists t\ge 0 \ :  \ X_t >L]=0\}.
\]
Since $U\subset K$, $K$ is nonempty. To show that $K$ is invariant for \eqref{eq:model}, let $x \in K$. By the Markov property,
\[
0=\Prob_x[\exists t\ge 0 \ :  \ X_t > L] \ge \E_x\Big[\Prob_y[\exists t\ge 0 \ :  \ X_t > L] \mathds{1}_{K^c}\Big].
\]
Since, $\Prob_y[\exists t\ge 0 \ :  \ X_t > L]>0$ for all $y\in K^c$, $\Prob_x[X_1 \in K^c]=0$. Hence, $K$ is a positive invariant set for \eqref{eq:model}. 

If $0\in U$, then it follows from the same arguments that $\{x\in[0,L) \ : \ \Prob_x[\exists t\ge 0 \ :  \ X_t >L]=0\} \supset U$ is invariant for \eqref{eq:model}.
\end{proof}

\textbf{Proof of Proposition \ref{prop:accessibility}.} Since $F_0$ is dissipative, there exists an attractor $A$ such that $\mathcal{B}(A)=\R_+$. Let $V\subset \R_+$ be a neighborhood of $A$ and $M_0 =\inf\{M>0 \ : \ V\subset [0,M] \}$. For any $M>M_0$, Proposition \ref{proposition-attractor} applies to $A$, $V$ and $[0,M]$. Hence there exists $\eps: [M_0,\infty) \rightarrow \R_+$ a decreasing function, $T: [M_0,\infty) \rightarrow \N$ such that, for every $M>M_0$, every $\eps(M)$ chain of length $t\ge T(M)$ starting in $[0,M]$ ends in $V$. We extend the functions $\eps$ to $\R_+$ by defining $\eps(M)=\eps(M_0)$ for all $M<M_0$.

Fix $M\ge 0$, and assume that \eqref{eq:model} is an $\eps(M)$-small noise system. If $M>M_0$, then Lemma~\ref{lemma:invariant} implies that there exists an invariant set $K_M\supset [0,M]$ for the system~\eqref{eq:model}. 

Assume that $F_0$ has no positive attractor. Propositions  \ref{proposition-chain} and \ref{proposition-att} imply that, for all $x\in \overline{K_M}$ (the closure of $K_M$) and all $\eps\le \eps(M)$, there exists a neighborhood $U_{x,\eps}$ of $x$ and $\beta_{x,\eps}>0$ such that
\[
\Prob_z[\exists t \ge 0 \ : \ X_t < \eps]>\beta_{x,\eps}
\] 
for all $z \in U_{x,\eps}$. Compactness of $\overline{K}_M$ implies that, for any $\eps \le \eps(M)$, there is $\beta_\eps>0$ such that
\[
\Prob_z[\exists t \ge 0 \ : \ X_t < \eps ]>\beta_\eps
\]
for all $z\in \overline{K}_M$. Hence $\{0\}$ is accessible from $K_M$.

If $M<M_0$, then $[0,M] \subset [0,M_0] \subset K_{M_0}$. Moreover, by definition of $\eps(M)$, $K_{M_0}$ is invariant for \eqref{eq:model} and $\{0\}$ is accessible from $K_{M_0}$. This concludes the proof.

\subsection{Proof of Theorem \ref{thm:persistence}}\label{proof:persistence}

After showing that the system \eqref{eq:model} is almost surely bounded, the almost surely persistence follows as in the proof of Theorem 1 in \cite{jmb-11}. Define $V: \R_+ \rightarrow \R_+$  by $V(x)=x$, $\alpha, \beta: \Omega \rightarrow \R_+$ by $\alpha(\xi)=\sup\{f(x,\xi) : x >x_c\}$ and $\beta(\xi) = \sup\{xf(x,\xi) : x \le x_c\}$. Hence, for any $\xi \in \Omega$ and $x\in \R_+$, we have $V(xf(x,\xi))\le \alpha(\xi)V(x) + \beta(\xi)$. By assumption, $\E[\ln \alpha]<0$ and, by continuity of $f$, $\E[\ln^+ \beta]<\infty$ where $\ln^+(x)=\max\{0,x\}$. Hence Proposition 4 in \cite{tpb-09} implies that the system \eqref{eq:model} is almost surely bounded.
%

\subsection{Proof of Proposition \ref{prop:invariant-set-compact}}\label{proof:invariant-set}

Assume $\E[\log f(0,\xi_1)]<0$. Recall, we say \eqref{eq:model} is unconditionally extinct if $\lim_{t\to\infty}X_t =0$ with probability one whenever $X_0=x\ge 0$. If there exists a positive invariant set $B\subset (0,M]$, then the system \eqref{eq:model} can not be unconditionally extinct. If the system \eqref{eq:model} is not unconditionally extinct, then, by Theorem \ref{thm:extinction}, $\{0\}$ is not accessible from $[0,M]$. Therefore there exists $\eps>0$ such that for all $n\ge 1$, there exists $x_n \in [0,M]$ such that
\[
\Prob_{x_n}[\exists t \ge 0 \ X_t < \eps]< \frac{1}{n}.
\]
Since the event $\{\exists t \ge 0 \ X_t < \eps\}$ is an open set of $\Omega$, compactness of $[0,M]$ and weak$^*$ continuity of $x \mapsto \Prob_x$ imply there exists $x \in [0,M]$ such that $\Prob_x[\exists t \ge 0 \ X_t < \eps]=0$. Define the non empty  positive set $B=\{x\in [0,M]\ : \ \Prob_x[\exists t \ge 0 \ X_t < \eps]=0\} \subset [\eps,M]$. To show that  $B$ is invariant, let $x\in B$. We will show that $\Prob_x[X_1\in B] =1$. By the Markov property,
\[
0=\Prob_x[\exists t \ge 0 \ X_t < \eps] \ge \E_x\Big[ \Prob_{X_1}[\exists t \ge 0 \mbox{ s.t. } X_t < \eps]\mathds{1}_{B^c}(X_1)\Big].
\]
Since $\Prob_y[\exists t \ge 0 \mbox{ s.t. } X_t < \eps]>0$ for all $y\in B^c$, $\Prob_x[X_1\in B^c]=0$ for all $x\in B$. Hence, $B$ is invariant. 

\subsection{Proof of Proposition \ref{prop:invariant-set}}\label{proof:invariant}

Assume that $A \subset (0,\infty)$ is a positive attractor with basin of attraction $B(A)$. Let $V\subset U \subset B(A)$ be positive compact neighborhoods of $A$. Proposition \ref{proposition-attractor} applies to $A$, $V$ and $U$. Hence, there exists $\eps$ and $T\ge 0$ such that every $\eps$ chain of length $t\ge T$ starting in $U$ ends in $V$. Assume the system \eqref{eq:model} satisfies \textnormal{\textbf{H.2}} for $\eps$. Hence, Lemma~\ref{lemma:invariant} implies that there exists a positive bounded invariant set for \eqref{eq:model} which concludes the proof.

\bibliographystyle{plainnat}

\bibliography{seb}

\end{document}